\documentclass[reqno]{amsart}

\newtheorem{thm}{Theorem}[section]
\newtheorem{cor}[thm]{Corollary}

\newtheorem{prop}[thm]{Proposition}
\theoremstyle{definition}

\theoremstyle{remark}
\newtheorem{rem}[thm]{Remark}

\begin{document}
\title[Ruled CR--submanifolds of l.c.K. manifolds]{Ruled CR--submanifolds of locally conformal K\"{a}hler manifolds}

\author[G.E. V\^\i lcu]{Gabriel Eduard V\^\i lcu}

\begin{abstract} The purpose of this paper is to study the canonical totally real foliations of
CR--submanifolds in a locally conformal K\"{a}hler manifold.\\
{\em 2010 Mathematics Subject Classification:} 53C15. \\
{\em Keywords:} locally conformal K\"{a}hler structure, ruled submanifold, CR--submanifold, distribution, foliation.
\end{abstract}
\maketitle \thispagestyle{empty}

\section{Introduction}

The concept of CR--submanifold, first introduced in K\"{a}hler geometry by A. Bejancu \cite{BEJ}, was later considered
and studied in locally conformal K\"{a}hler ambient by many authors (see e.g. \cite{BAR,BM,CF,DR,DR3,DS,MATS,MUN,ORN,PAP,SG}).
Such a submanifold
comes naturally equipped with some canonical foliations, which were first investigated by B.Y. Chen and P. Piccinni \cite{CP} (see also Chapter 12 from the monograph \cite{DO}). One of these foliations, denoted by $\mathfrak{F}^\perp$ and called the totally real foliation, is given by the totally real distribution involved in the definition of the CR--submanifold, proven to be always completely integrable by D.E. Blair and B.Y. Chen \cite{BLCH}. On the other hand, A. Bejancu and H.R. Farran \cite[Chapter 5]{BJCF} investigated the relationship between the geometry of the totally real foliation on a CR--submanifold of a K\"{a}hler manifold and the geometry of the CR--submanifold itself, stressing on the links between the foliation and the complex structure on the embedding manifold (see also the monograph \cite{BDD} for an excellent survey concerning foliations in CR geometry). Moreover, they also used the theory of ruled submanifolds (see \cite{ROV} for a detailed survey on the topic) to characterize some classes of CR--submanifolds in K\"{a}hler manifolds. At the end of the chapter, the authors have proposed, as an interesting and useful research, the extension of this study to CR--submanifolds of manifolds endowed with various geometric structures. This was done recently for quaternionic and paraquaternionic K\"{a}hler ambient \cite{IIV,IMAV,VIL}.  In this paper, following the same techniques, we study the CR--submanifolds in a locally conformal K\"{a}hler manifold. In particular, we obtain necessary and sufficient conditions for a CR--submanifold of a  locally conformal K\"{a}hler manifold to be ruled with respect to the totally real foliation $\mathfrak{F}^\perp$. In the last part of the paper characterizations are provided for this foliation to become Riemannian, i.e. with bundle--like metric.

\section{Preliminaries}

Let $(\overline{M},J,\overline{g})$ be an almost Hermitian manifold of dimension $2n$, where $J$ denotes the almost complex structure
and $\overline{g}$ the Hermitian metric. Then $(\overline{M},J,\overline{g})$ is called a \emph{locally conformal K\"{a}hler} (briefly l.c.K.)
manifold if for each point $p$ of
$\overline{M}$ there exists an open neighbourhood $U$ of $p$ and a positive function $f_U$ on $U$ so
that the local metric \[\overline{g}_U = {\rm exp }(-f_U)\overline{g}_{|U}\]
is K\"{a}hlerian (see \cite{L,V}). If $U = \overline{M}$, then the manifold $(\overline{M},J,\overline{g})$
is said to be a \emph{globally conformal K\"{a}hler} (briefly g.c.K.) manifold. Equivalently (see \cite{DO}), $(\overline{M},J,\overline{g})$
is l.c.K. if and only if there exists a closed 1--form $\omega$, globally
defined on $\overline{M}$, such that
\[
d\Omega=\omega\wedge\Omega,
\]
where $\Omega$ is the K\"{a}hler 2--form associated with $(J,\overline{g})$, i.e. \[\Omega(X,Y)=\overline{g}(X,JY),\] for $X,Y\in\Gamma(T\overline{M})$.
The 1--form $\omega$ is called the \emph{Lee form} and its metrically equivalent vector field $B=\omega^\sharp$, where $\sharp$ means the rising
of the indices with respect to $\overline{g}$, namely
\[
\overline{g}(X,B)=\omega(X),
\]
for all $X\in\Gamma(T\overline{M})$, is called \emph{Lee vector field}. It is known that $(\overline{M},J,\overline{g})$ is globally conformal K\"{a}hler
(respectively K\"{a}hler) if the Lee-form $\omega$ is exact (respectively $\omega=0$). It is also known that Levi--Civita connections $D^U$ of the local metrics $\overline{g}_U$ glue up to a globally defined torsion free
linear connection $D$ on $\overline{M}$, called the \emph{Weyl connection} of the l.c.K. manifold $\overline{M}$, given by
\[
D_XY=\overline{\nabla}_XY-\frac{1}{2}\left[\omega(X)Y+\omega(Y)X-\overline{g}(X,Y)B\right]
\]
for any $X,Y\in\Gamma(T\overline{M})$, where $\overline{\nabla}$ is the Levi--Civita connection of $\overline{g}$. Moreover, Weyl connection $D$ satisfies
$D\overline{g}=\omega\otimes\overline{g}$ and $DJ=0$. As a consequence, considering the anti--Lee form $\theta=\omega\circ J$
and the anti--Lee vector field $A=-JB$, one can obtain a third equivalent definition in terms of the Levi--Civita connection $\overline{\nabla}$ of the metric $\overline{g}$ (see \cite{DO}). Namely, $(\overline{M},J,\overline{g})$ is l.c.K. if and only if the
following equation is satisfied for any $X,Y\in\Gamma(T\overline{M})$:
\begin{equation}\label{LC}
(\overline{\nabla}_XJ)Y=\frac{1}{2}\left[\theta(Y)X-\omega(Y)JX-\overline{g}(X,Y)A-\Omega(X,Y)B\right].
\end{equation}

A submanifold $M$ of a l.c.K. manifold $(\overline{M},J,\overline{g})$ is called a \emph{CR--submanifold} if there exists a differentiable distribution $D:p\rightarrow D_p\subset T_pM$ on $M$ satisfying the following conditions:
\begin{enumerate}
  \item[i.] $D$ is holomorphic, i.e. $JD_p$=$D_p$, for each $p\in M$;
  \item[ii.] the complementary orthogonal distribution $D^{\perp}: p\rightarrow D_p^{\perp}\subset T_pM$
is totally real, i.e. $JD_p^{\perp}\subset T_p^{\perp}M$ for each $p\in M$.
\end{enumerate}

If ${\rm dim } D_p^\perp=0$ (resp. ${\rm dim } D_p=0$), then the CR--submanifold is
said to be a \emph{holomorphic} (resp. a \emph{totally real}) submanifold.  A CR--submanifold is
called a \emph{proper} CR--submanifold if it is neither holomorphic nor totally
real.

\begin{rem}
Let $M$ be a CR--submanifold of a l.c.K. manifold $(\overline{M},J,\overline{g})$. By the definition of
a CR--submanifold we have the orthogonal decomposition
$$TM=D\oplus D^\perp.$$

Also, the normal bundle has the orthogonal decomposition
$$TM^\perp=JD^\perp\oplus\mu,$$
where $\mu$ is the subbundle of the normal bundle $TM^\perp$ which
is the orthogonal complement of $JD^\perp$. Corresponding to the last decomposition, any normal vector field $N$ can be written as
$N=N_{JD^\perp}+N_{\mu}$, where $N_{JD^\perp}$ (resp. $N_\mu$) is the $JD^\perp$-- (resp. $\mu$--) component of $N$.
It is easy to see that the subbundle $\mu$ is invariant under the action of $J$. We note that if $\mu={0}$, then the CR-submanifold is said to be an \emph{anti--holomorphic}
submanifold or a \emph{generic} submanifold.

If we denote by $\nabla$ the Levi--Civita connection on $(M,g)$, where $g$ is the induced Riemannian metric by $\overline{g}$ on $M$, then the Gauss and Weingarten formulas are given by:
\begin{equation}\label{G}
       \overline{\nabla}_XY=\nabla_XY+h(X,Y),\ \forall X,Y \in
\Gamma(TM)
       \end{equation}
and
\begin{equation}\label{W}
       \overline{\nabla}_XN=-a_NX+\nabla_{X}^{\perp}N,\ \forall X\in
\Gamma(TM),\ \forall N\in \Gamma(TM^\perp)
       \end{equation}
where $h$ is the second fundamental form of $M$, $\nabla^\perp$ is
the connection on the normal bundle and $a_N$ is the shape operator
of $M$ with respect to $N$. It is well--known that $h$ is a symmetric $F(M)$--bilinear form and $a_N$ is a self--adjoint operator, related by:
\begin{equation}\label{GW}
       g(a_NX,Y)=\overline{g}(h(X,Y),N)
       \end{equation}
for all $ X,Y\in \Gamma(TM)$ and $N\in \Gamma(TM^\perp)$.
We say (see \cite{BJCF}) that the distribution $D$ (resp. $D^\perp$) is $a_N$--\emph{invariant}, if $a_NX\in\Gamma(D)$ (resp. $a_NZ\in\Gamma(D^\perp)$)
for any $X\in \Gamma(D)$ (resp. $Z\in\Gamma(D^\perp)$).

A CR--submanifold $M$ of a l.c.K. manifold $(\overline{M},J,\overline{g})$ is called:
\begin{enumerate}
  \item[i.] $D$-\emph{geodesic} if $h(X,Y)=0$, $\forall X,Y\in\Gamma(D)$.
  \item[ii.] $D^\perp$-\emph{geodesic} if $h(X,Y)=0$, $\forall X,Y\in\Gamma(D^\perp)$.
  \item[iii.] \emph{mixed geodesic} if $h(X,Y)=0$, $\forall X\in\Gamma(D), Y\in\Gamma(D^\perp)$.
\end{enumerate}
\end{rem}

We recall now the following result which we shall need in the sequel.

\begin{thm}\label{T1}
Let $M$ be a CR--submanifold of a l.c.K. manifold $(\overline{M},J,\overline{g})$. Then:
\begin{enumerate}
  \item[i.] The totally real distribution $D^\perp$ is integrable \cite{BLCH}.
  \item[ii.] The holomorphic distribution $D$ is integrable if and only if
  \[
  \overline{g}(h(X,JY),JZ)=\overline{g}(h(JX,Y),JZ)-\Omega(X,Y)\theta(Z)
  \]
for all $X,Y\in\Gamma(D)$ and $Z\in\Gamma(D^\perp)$ \cite{BD}.
\end{enumerate}
\end{thm}

\section{Totally real foliation of a CR--submanifold in a locally conformal K\"{a}hler manifold}

Let $M$ be a CR--submanifold of a l.c.K. manifold $(\overline{M},J,\overline{g})$. From Theorem \ref{T1} we have that the
distribution $D^\perp$ is always integrable and gives rise to a foliation of $M$ by totally-real submanifolds of $\overline{M}$.
So any CR--submanifold of a l.c.K. manifold comes naturally equipped with a foliation denoted by $\mathfrak{F}^\perp$ and called the \emph{totally real foliation}.
We note that if the holomorphic distribution $D$ is also integrable, then $M$ carries a foliation by holomorphic submanifolds of $\overline{M}$, called the \emph{Levi foliation} (see \cite{BD,DN}).

We recall that if each leaf of a foliation $\mathfrak{F}$ on $M$ is a
totally geodesic submanifold of $M$, then we say that $\mathfrak{F}$ is a \emph{totally geodesic
foliation}. Next we state some characterizations of totally geodesic totally real foliations
on CR--submanifolds.

\begin{prop}\label{T2}
The canonical totally real foliation  $\mathfrak{F}^\perp$ on
a CR--submanifold $M$ of a l.c.K. manifold
$(\overline{M},J,\overline{g})$ is  a totally geodesic
foliation if and only if
\begin{equation}\label{rel}
\theta(Y)JX=2h_{JD^\perp}(X,Y),\ \forall X\in\Gamma(D^\perp),\ Y\in\Gamma(D).
\end{equation}
\end{prop}
\begin{proof}
For $X,Z\in\Gamma(D^\perp)$ and $Y\in\Gamma(D)$, using (\ref{LC})--(\ref{GW}), we derive:
\begin{eqnarray}
\overline{g}(J\nabla_XZ,Y)&=&-\overline{g}(\nabla_XZ,JY)\nonumber\\
                            &=&-\overline{g}(\overline{\nabla}_XZ-h(X,Z),JY)\nonumber\\
                            &=&\overline{g}(-(\overline{\nabla}_XJ)Z+\overline{\nabla}_XJZ,Y)\nonumber\\
                            &=&-\frac{1}{2}\overline{g}(\theta(Z)X-\omega(Z)JX-g(X,Z)A-\Omega(X,Y)B-2\overline{\nabla}_XJZ,Y)\nonumber\\
                            &=&-\frac{1}{2}\overline{g}(g(X,Z)JB-2\overline{\nabla}_XJZ,Y)\nonumber\\
                            &=&\frac{1}{2}g(X,Z)\overline{g}(B,JY)+\overline{g}(-a_{JZ}X+\nabla_X^{\perp}JZ,Y)\nonumber\\
                            &=&\frac{1}{2}g(X,Z)\omega(JY)-\overline{g}(h(X,Y),JZ)\nonumber\\
                            &=&\frac{1}{2}\overline{g}(\theta(Y)JX,JZ)-\overline{g}(h(X,Y),JZ).\nonumber
       \end{eqnarray}

Therefore, we obtain
\begin{equation}\label{2.1}
\overline{g}(J\nabla_XZ,Y)=\frac{1}{2}\overline{g}(\theta(Y)JX-2h_{JD^\perp}(X,Y),JZ),\ \forall X,Z\in\Gamma(D^\perp),\ Y\in\Gamma(D).
\end{equation}

If we suppose now $\mathfrak{F}^\perp$ is  a totally geodesic
foliation, then $\nabla_XZ\in\Gamma(D^\perp)$, for all $X,Z\in\Gamma(D^\perp)$, and from
(\ref{2.1}) we deduce:
\[
\overline{g}(\theta(Y)JX-2h_{JD^\perp}(X,Y),JZ)=0,\ \forall Z\in\Gamma(D^\perp)
\]
and the implication follows.

Conversely, if we suppose $\theta(Y)JX=2h_{JD^\perp}(X,Y)$, for all $X\in\Gamma(D^\perp)$, $Y\in\Gamma(D)$, then from (\ref{2.1}) we derive:
\[\overline{g}(J\nabla_XZ,Y)=0\]
and we conclude $\nabla_XZ\in\Gamma(D^\perp)$. Thus
$\mathfrak{F}^\perp$ is a totally geodesic foliation.
\end{proof}

\begin{rem}
An alternative proof of the above Proposition can be obtained using \cite[Lemma 1, p. 343]{BD}.
\end{rem}

\begin{thm}\label{T3}
Let $M$ be a CR--submanifold of a l.c.K. manifold $(\overline{M},J,\overline{g})$ such that the Lee vector field $B$ is normal to $M$. Then the next assertions are equivalent:
\begin{enumerate}
  \item[i.] The canonical totally real foliation  $\mathfrak{F}^\perp$ on $M$ is totally geodesic.
  \item[ii.] $h(X,Y)\in\Gamma(\mu)$, $\forall X\in\Gamma(D^\perp)$, $Y\in\Gamma(D)$.
  \item[iii.] The totally real distribution $D^\perp$ is $a_N$--invariant for any $N\in\Gamma(JD^\perp)$.
  \item[iv.] The holomorphic distribution $D$ is $a_N$--invariant for any $N\in\Gamma(JD^\perp)$.
\end{enumerate}
\end{thm}
\begin{proof}
Since $B$ is normal to $M$, we deduce
\[
\theta(Y)=\omega(JY)=\overline{g}(JY,B)=0
\]
for any $Y\in\Gamma(D)$. Therefore, from the above Proposition we obtain (i) $\Leftrightarrow$ (ii).

The equivalence of (ii) and (iii) follows easily from (\ref{GW}), while
the equivalence of (iii) and (iv) holds because $a_N$ is a self--adjoint operator.
\end{proof}

\begin{rem}
We note that Theorem \ref{T3} extends Theorem 4.1 in \cite[p. 247]{BJCF} from the case of an ambient K\"{a}hlerian
manifold to the case of an ambient 1.c.K. manifold.
\end{rem}

\begin{cor}\label{C4}
Let $M$ be a CR--submanifold of a l.c.K. manifold $(\overline{M},J,\overline{g})$ such that the Lee vector field $B$ is normal to $M$. Then:
\begin{enumerate}
  \item[i.] If $M$ is mixed geodesic, then the totally real foliation  $\mathfrak{F}^\perp$ on $M$ is totally geodesic.
  \item[ii.] If $M$ is an anti--holomorphic submanifold, then $M$ is mixed geodesic if and only if the totally real foliation  $\mathfrak{F}^\perp$ is totally geodesic.
\end{enumerate}
\end{cor}
\begin{proof}
The proof is clear from Theorem \ref{T3}.
\end{proof}

\begin{rem}
We note that the Corollary \ref{C4}(i.) has been also obtained using a different proof by Dragomir \cite{DR}
(see also \cite[Theorem 12.6, p. 168]{DO}). On the other hand, Corollary \ref{C4}(ii.) gives us an interesting geometric characterization of mixed
geodesic anti--holomorphic submanifolds in a l.c.K. manifold normal to the Lee vector field.
Thus, $M$ is mixed geodesic if and
only if any geodesic of a leaf of $D^\perp$ is a geodesic of $M$. On another hand, according to
Corollary \ref{C4}(i.), if $M$ is totally geodesic, then $M$ is mixed geodesic and any geodesic of a leaf of $\mathfrak{F}^\perp$ is
a geodesic of $M$ which in turn is a geodesic of $\overline{M}$. Therefore any leaf of $\mathfrak{F}^\perp$ is totally geodesic
immersed in $(\overline{M},J,\overline{g})$. It is important to note that this property is also true in K\"{a}hler ambient (see \cite[Corollary 4.4, p. 148]{BJCF}).
\end{rem}

A submanifold $M$ of a Riemannian manifold
$(\overline{M},\overline{g})$ is said to be a \emph{ruled
submanifold} if it admits a foliation whose leaves are totally
geodesic immersed in $(\overline{M},\overline{g})$. A CR--submanifold which is a ruled submanifold with
respect to the canonical foliation $\mathfrak{F}^\perp$  is called a \emph{totally real ruled} CR--submanifold.
We are able now to state the following characterization of totally real ruled CR--submanifolds in l.c.K. manifolds.

\begin{thm}\label{T4}
Let $M$ be a CR--submanifold of a l.c.K. manifold $(\overline{M},J,\overline{g})$. Then the next assertions are equivalent:
\begin{enumerate}
  \item[i.] $M$ is a totally real ruled CR--submanifold.
  \item[ii.] $M$ is $D^\perp$-geodesic and the anti--Lee form $\theta$ and the second fundamental form $h$ of the submanifold are related by (\ref{rel}).
  \item[iii.] The second fundamental form $h$, the anti--Lee form $\theta$ and the anti--Lee vector field $A$ are related by (\ref{rel}) and satisfy:
                \begin{equation}\label{e7}
                h(X,Z)\in\Gamma(\mu),\ \forall X,Z\in\Gamma(D^\perp)
                \end{equation}
                and
                \begin{equation}\label{e8}
                \left(\nabla_X^\perp JZ\right)_\mu=-\frac{1}{2}g(X,Z)A_\mu,\ \forall X,Z\in\Gamma(D^\perp),
              \end{equation}
  where the index $\mu$ denotes the $\mu$--component of the vector field.
\end{enumerate}
\end{thm}
\begin{proof}
i. $\Leftrightarrow$ ii. For any $X,Z\in\Gamma(D^\perp)$ we have:
\begin{eqnarray}
       \overline{\nabla}_XZ&=&\nabla_XZ+h(X,Z)\nonumber\\
       &=&\nabla_X^{D^\perp}Z+h^{D^\perp}(X,Z)+h(X,Z)\nonumber
       \end{eqnarray}
and thus we conclude that the leafs of $D^\perp$ are totally
geodesic immersed in $\overline{M}$ if and only if $h^{D^\perp}=0$ and $M$ is
$D^\perp$-geodesic.
The equivalence follows now easily from Proposition \ref{T2}.\\
i. $\Leftrightarrow$ iii. For  $X,Z\in\Gamma(D^\perp)$, and
$U\in\Gamma(D)$ we obtain similarly as in the proof of Proposition \ref{T2}:
\begin{eqnarray}\label{e9}
\overline{g}(\overline{\nabla}_XZ,U)&=&\overline{g}(J\overline{\nabla}_XZ,JU)\nonumber\\
&=&\overline{g}(-(\overline{\nabla}_XJ)Z+\overline{\nabla}_XJZ,JU)\nonumber\\
&=&\frac{1}{2}\overline{g}(\theta(JU)JX-2h_{JD^\perp}(X,JU),JZ).
       \end{eqnarray}

On the other hand, if $X,Z,W\in\Gamma(D^\perp)$, then taking account of (\ref{G}) we deduce:
\begin{eqnarray}\label{e10}
       \overline{g}(\overline{\nabla}_XZ,JW)&=&\overline{g}(\nabla_XZ+h(X,Z),JW)\nonumber\\
       &=&\overline{g}(h(X,Z),JW).
       \end{eqnarray}

If we consider now $X,Z\in\Gamma(D^\perp)$ and $N\in\Gamma(\mu)$, then making use of (\ref{LC}) and (\ref{W}) we derive:
\begin{eqnarray}
\overline{g}(\overline{\nabla}_XZ,N)&=&\overline{g}(J\overline{\nabla}_XZ,JN)\nonumber\\
&=&\overline{g}(-(\overline{\nabla}_XJ)Z+\overline{\nabla}_XJZ,JN)\nonumber\\
       &=&-\frac{1}{2}\overline{g}((\theta(Z)X-\omega(Z)JX-g(X,Z)A-\Omega(X,Z)B)-2\overline{\nabla}_XJZ,JN)\nonumber\\
       &=&\frac{1}{2}\overline{g}(g(X,Z)A+2\overline{\nabla}_XJZ,JN)\nonumber\\
       &=&\frac{1}{2}\overline{g}(g(X,Z)A+2\nabla_X^{\perp}JZ,JN)\nonumber
       \end{eqnarray}
and thus we obtain:
\begin{equation}\label{e11}
\overline{g}(\overline{\nabla}_XZ,N)=\frac{1}{2}\overline{g}(g(X,Z)A_\mu+2\left(\nabla^\perp_X
JZ\right)_\mu,JN).
\end{equation}

Finally, $M$ is a totally real ruled CR--submanifold of $(\overline{M},J,\overline{g})$ if and only if
$\overline{\nabla}_X Z\in\Gamma(D^\perp)$, $\forall
X,Z\in\Gamma(D^\perp)$ and by using (\ref{e9}), (\ref{e10}) and
(\ref{e11}) we deduce the equivalence.
\end{proof}

\begin{cor}\label{c8}
If $M$ is a CR--submanifold of a l.c.K. manifold $(\overline{M},J,\overline{g})$ such that $B\in\Gamma(JD^\perp)$, then the next assertions
are equivalent:
\begin{enumerate}
  \item[i.] $M$ is a totally real ruled CR--submanifold.
  \item[ii.] $M$ is $D^\perp$-geodesic and the second fundamental form
satisfies \[h(X,Y)\in\Gamma(\mu),\ \forall X\in\Gamma(D^\perp),\ Y\in\Gamma(D).\]
  \item[iii.] The subbundle $JD^\perp$ is $D^\perp$-parallel, i.e:
                \[\nabla_X^\perp JZ\in\Gamma(JD^\perp),\ \forall X,Z\in \Gamma(D^\perp)\]
                and the second fundamental form satisfies
                \[h(X,Y)\in\Gamma(\mu),\ \forall X\in\Gamma(D^\perp),\ Y\in\Gamma(TM).\]
  \item[iv.]  The shape operator satisfies
                \[a_{JZ}X=0,\ \forall X,Z\in \Gamma(D^\perp)\]
and
                \[a_N X\in\Gamma(D),\ \forall X\in\Gamma(D^\perp),\ N\in\Gamma(\mu).\]
\end{enumerate}
\end{cor}
\begin{proof}
The equivalence of (i.), (ii.) and (iii.) is clear from the above theorem, since for any $Y\in\Gamma(D)$ we have
\[
\theta(Y)=g(JY,B)=0.
\]

The equivalence of (ii) and (iii.) follows from (\ref{GW}).
\end{proof}

\begin{cor}\label{c9}
Let $M$ be a CR--submanifold of a l.c.K. manifold $(\overline{M},J,\overline{g})$ such that Lee vector field $B$ is normal to $M$.
If $M$ is totally geodesic, then $M$ is a totally real ruled CR--submanifold.
\end{cor}
\begin{proof}
The assertion is clear from Theorem \ref{T4}.
\end{proof}

\section{Foliations with bundle--like metric on CR--submanifolds of locally conformal K\"{a}hler manifolds}

Let $(M,g)$ be a Riemannian manifold and $\mathfrak{F}$ a foliation
on $M$. The metric $g$ is said to be bundle--like for the foliation
$\mathfrak{F}$ if the induced metric on the transversal distribution
$\mathcal{D}^\perp$ is parallel with respect to the intrinsic connection on
$\mathcal{D}^\perp$. This is true if and only if the Levi--Civita connection
$\nabla$ of $(M,g)$ satisfies (see \cite{BJCF}):
\begin{equation}\label{bl}
g(\nabla_{Q^\perp Y}QX,Q^\perp Z)+g(\nabla_{Q^\perp Z}QX,Q^\perp
Y)=0,\ \forall X,Y,Z\in\Gamma(TM),
\end{equation}
where $Q^\perp$ (resp. $Q$) is the projection morphism of $TM$ on $\mathcal{D}^\perp$ (resp $D$).

If for a given foliation $\mathfrak{F}$ there exists a Riemannian
metric $g$ on $M$ which is bundle-like for $\mathfrak{F}$, then we
say that $\mathfrak{F}$ is a Riemannian foliation on $(M,g)$.

In what follows we provide necessary
and sufficient conditions for the induced metric on a CR--submanifold of a l.c.K. manifold to be bundle--like for the totally real
foliation $\mathfrak{F}^\perp$.

\begin{thm}\label{T5}
If $M$ is a CR--submanifold of a l.c.K. manifold $(\overline{M},J,\overline{g})$, then the next assertions
are equivalent:
\begin{enumerate}
  \item[i.] The induced metric $g$ on $M$ is bundle--like for the canonical totally real
foliation $\mathfrak{F}^\perp$.
  \item[ii.] The second fundamental form $h$ of the submanifold and anti--Lee vector field $A$ satisfy:
\[g(U,V)A+h(U,JV)+h(V,JU)\in\Gamma(TM)\oplus\Gamma(\mu),\] for any $U,V\in\Gamma(D)$.
\end{enumerate}
\end{thm}
\begin{proof}
From (\ref{bl}) we deduce that $g$ is bundle-like for the canonical totally real
foliation $\mathfrak{F}^\perp$ if and only if:
\begin{equation}\label{e13}
g(\nabla_U X,V)+g(\nabla_V X,U)=0,\ \forall X\in\Gamma(D^\perp),\
U,V\in\Gamma(D).
\end{equation}

On the other hand, using (\ref{LC})-(\ref{GW}), we obtain for any $X\in\Gamma(D^\perp),\ U,V\in\Gamma(D)$:
\begin{eqnarray}
       g(\nabla_U X,V)&+&g(\nabla_V
X,U)=\overline{g}(\overline{\nabla}_U X-h(U,X),V)+
       \overline{g}(\overline{\nabla}_V X-h(V,X),U)\nonumber\\
       &=&\overline{g}(\overline{\nabla}_U X,V)+
       \overline{g}(\overline{\nabla}_V X,U)\nonumber\\
       &=&\overline{g}(-(\overline{\nabla}_U
J)X+\overline{\nabla}_U JX,JV)+\overline{g}(-(\overline{\nabla}_V
J)X+\overline{\nabla}_V JX,JU)\nonumber\\
       &=&-\frac{1}{2}\overline{g}(\theta(X)U-\omega(X)JU-g(U,X)A-\Omega(U,X)B
       -2\overline{\nabla}_U JX,JV)\nonumber\\
       &&-\frac{1}{2}\overline{g}(\theta(X)V-\omega(X)JV-g(V,X)A-\Omega(V,X)B
       -2\overline{\nabla}_V JX,JU)\nonumber\\
       &=&-\frac{1}{2}\overline{g}(\theta(X)U-\omega(X)JU
       -2\overline{\nabla}_U JX,JV)\nonumber\\
       &&-\frac{1}{2}\overline{g}(\theta(X)V-\omega(X)JV
       -2\overline{\nabla}_V JX,JU)\nonumber\\
       &=&\omega(X)g(U,V)+\overline{g}(\overline{\nabla}_U JX,JV)+\overline{g}(\overline{\nabla}_V JX,JU)\nonumber\\
       &=&\omega(X)g(U,V)-g(A_{JX}U,JV)-g(A_{JX}V,JU)\nonumber\\
       &=&\omega(X)g(U,V)-\overline{g}(h(U,JV),JX)-\overline{g}(h(V,JU),JX).\nonumber
       \end{eqnarray}
and taking into account that $B=\omega^\sharp$ and $A=-JB$ we derive:
\begin{equation}\label{e14}
g(\nabla_U X,V)+g(\nabla_V X,U)=-\overline{g}(g(U,V)A+h(U,JV)+h(V,JU),JX).
\end{equation}

The proof is now complete from (\ref{e13}) and (\ref{e14}).
\end{proof}

\begin{cor}
Let $M$ be a CR--submanifold of a l.c.K. manifold $(\overline{M},J,\overline{g})$.
\begin{enumerate}
  \item[i.] If $B\in\Gamma(D)\oplus\Gamma(TM^\perp)$, then the induced metric $g$ on $M$ is bundle--like for the canonical totally real
foliation $\mathfrak{F}^\perp$ if and only if \[h(U,JV)+h(V,JU)\in\Gamma(\mu),\ \forall U,V\in\Gamma(D).\]
  \item[ii.] If $B$  has a non--vanishing component in $\Gamma(D^\perp)$, then the induced metric $g$ on $M$ is not bundle--like for the canonical totally real foliation $\mathfrak{F}^\perp$.
\end{enumerate}
\end{cor}
\begin{proof}
The proof follows from Theorem \ref{T5}.
\end{proof}

\begin{cor}
If $M$ is an anti-holomorphic submanifold of a l.c.K. manifold $(\overline{M},J,\overline{g})$, normal to the Lee field
of $\overline{M}$, then the induced metric $g$ on $M$ is bundle--like for the canonical totally real
foliation $\mathfrak{F}^\perp$ if and only if \[h(U,JV)+h(V,JU)=0,\ \forall U,V\in\Gamma(D).\]
\end{cor}
\begin{proof}
The assertion is clear from the above Corollary.
\end{proof}

\section*{Acknowledgement} This work was supported by
CNCS-–UEFISCDI, project number PN--II--ID--PCE--2011--3--0118.

Gabriel Eduard V\^{I}LCU$^{1,2}$ \\ \\
      $^1$University of Bucharest,\\
      Research Center in Geometry, Topology and Algebra,\\
      Str. Academiei, Nr. 14, Sector 1,\\
      Bucure\c sti 70109-ROMANIA\\
      e-mail: gvilcu@gta.math.unibuc.ro\\ \\
      $^2$Petroleum-Gas University of Ploie\c sti,\\
      Department of Mathematical Economics,\\
      Bd. Bucure\c sti, Nr. 39,\\
      Ploie\c sti 100680-ROMANIA\\
      e-mail: gvilcu@upg-ploiesti.ro

\end{document}